\documentclass[pdflatex,sn-mathphys-num]{sn-jnl}

\usepackage{ragged2e}
\usepackage{geometry}
\usepackage{graphicx}%
\usepackage{multirow}%
\usepackage{amsmath,amssymb,amsfonts}%
\usepackage{amsthm}%
\usepackage{mathrsfs}%
\usepackage[title]{appendix}%
\usepackage{xcolor}%
\usepackage{tikz-cd}
\usepackage{textcomp}%
\usepackage{manyfoot}%
\usepackage{booktabs}%
\usepackage{algorithm}%
\usepackage{algorithmicx}%
\usepackage{algpseudocode}%
\usepackage{ragged2e}
\usepackage{listings}%
\usepackage{etoolbox}  
\makeatletter
\patchcmd{\@maketitle}{\artauthors}{{\artauthors}}{\centering}{}
\makeatother

\theoremstyle{thmstyleone}%

\newtheorem{theorem}{Theorem}[section]
\usepackage{thmtools, thm-restate}
\newtheorem{conjecture}[theorem]{Conjecture}
\newtheorem{definition}[theorem]{Definition}

\newtheorem{lemma}[theorem]{Lemma}

\newtheorem{proposition}[theorem]{Proposition}

\newtheorem{remark}[theorem]{Remark}
\theoremstyle{thmstylethree}%
\begin{document}
\renewcommand\labelenumi{(\theenumi)}
\title[Article Title]{\textbf{The $\mu-$invariant of fine Selmer groups associated to general Drinfeld modules}}

\author{Hang Chen}


\abstract{ Let \textit{F} be a global function field over the finite field $\mathbb{F}_q$ where $q$ is a prime power and $A$ be the ring of elements in $F$ regular outside $\infty$. Let $\phi$ be an arbitrary Drinfeld module over \textit{F}. For a fixed non-zero prime ideal $\mathfrak{p}$ of \textit{A}, we show that on the constant $\mathbb{Z}_{\textit{p}}-$extension $\mathcal{F}$ of \textit{F}, the Pontryagin dual of the fine Selmer group associated to the $\mathfrak{p}-$primary torsion of $\phi$ over $\mathcal{F}$ is a finitely generated Iwasawa module such that its Iwasawa $\mu-$invariant vanishes. This provides a generalization of the results given in \cite{bib1}.}

\keywords{Iwasawa module, Selmer group, Drinfeld module, Galois cohomology}



\maketitle
\section{Introduction}\label{sec1}
\justifying
Let $K$ be a number field and $p$ be a prime number. A $\mathbb{Z}_p-$extension $\textbf{K}/K$ is the direct limit of a sequence of finite Galois extensions as below
\begin{align*}
    K = K_0 \subseteq K_1 \subseteq \cdots\subseteq K_n \subseteq K_{n+1} \subseteq \cdots \subseteq \textbf{K}
\end{align*} such that $\text{Gal}(K_n/K) \simeq \mathbb{Z}/p^n\mathbb{Z}$ for all $n$. If we denote the $p-$Sylow subgroup of the class group of $K_n$ by $\text{Cl}_p(K_n)$, then the classical Iwasawa theory tells us that for sufficiently large $n$, we have that
\begin{align} \label{p-exponent}
   \#\text{Cl}_p(K_n) = p^{\lambda n + \mu p^n + \nu}
\end{align} where $\lambda, \mu \in \mathbb{Z}_{\geq 0}$ and $\nu \in \mathbb{Z}$ are independent on $n$ (cf. \cite[\textbf{Theorem}~\textbf{13.13}]{bib3}) and these three integers are uniquely associated to $\textbf{K}$. Moreover, the inverse limit 
\begin{align*}
    X = \varprojlim_n \operatorname{Cl}_p(K_n)
\end{align*} is a finitely generated and torsion $\mathbb{Z}_p[\left[T\right]]-$module. Furthermore, Iwasawa made the following conjecture (cf. \cite{bib7}).
\begin{conjecture} [\textbf{Iwasawa}] For the cyclotomic $\mathbb{Z}_p-$extension $\textbf{K}$ of any number field $K$, the $\mu-$invariant vanishes.
\end{conjecture} \noindent If this conjecture is true, it implies that the inverse limit $X$ above is pseudo-isomorphic (cf. \cite[p.272]{bib3}) to a free $\mathbb{Z}_p-$module of rank $\lambda$ where $\lambda$ is from the formula $(\ref{p-exponent})$. This conjecture is solved by Ferrero and Washington if $K/\mathbb{Q}$ is abelian (cf. \cite{bib4}). \\[5mm] Coates and Sujatha (cf.\cite[Conjecture A]{bib6}) formulated an analogue of this conjecture in the context of the fine Selmer groups associated to elliptic curves over number fields: given an elliptic curve $E$ over a number field $K$, we denote the Pontryagin dual of the fine Selmer group of $E$ with respect to $\textbf{K}$ by $Y(E/\textbf{K})$ where $\textbf{K}$ is the cyclotomic $\mathbb{Z}_p-$extension of $K$ (cf. \cite[(42), (47)]{bib6}). 
\begin{conjecture} [\textbf{Coates and Sujatha}] For all elliptic curves $E$ over a number field $K$, $Y(E/\emph{\textbf{K}})$ is a finitely generated $\mathbb{Z}_p-$module.
\end{conjecture}
\noindent Ray (cf. \cite{bib1}) reformulated this conjecture over global function fields by considering fine Selmer groups associated to Drinfeld $\mathbb{F}_q[T]-$modules. Namely, Ray equipped the Pontryagin dual of a fine Selmer group with an Iwasawa module structure and showed that it is finitely generated and its $\mu-$invariant vanishes. \\[5mm] In our work, we aim to generalize Ray's result (cf.\cite{bib1}) by considering Drinfeld modules $\phi$ over any global function field $F$. Let $A$ be the ring of elements in $F$ regular outside a fixed place $\infty$. Given a fixed a non-zero prime ideal $\mathfrak{p}$ of $A$, we denote the unique place of $F$ corresponding to the prime ideal $\mathfrak{p}$ also by $\mathfrak{p}$. Consider the set $S$ 
where
\begin{align}\label{S}
    S:= \{ \mathfrak{p}, \infty \} \cup \{ \text{places corresponding to the bad reductions of } \phi\} .
\end{align}
Denote the constant $\mathbb{Z}_p-$extension of $F$ by $\mathcal{F}$ and denote the union of all the $\mathfrak{p}^n-$torsion of $\phi$ by $\phi[\mathfrak{p}^{\infty}]$. The fine Selmer group $\text{Sel}^S_0(\phi[\mathfrak{p}^{\infty}]/\mathcal{F})$ is a subgroup of the first Galois cohomology group $H^1(\mathcal{F}^{\text{sep}}/\mathcal{F}, \phi[\mathfrak{p^{\infty}}])$ which consists of the elements being trivial when restricting to the decomposition groups with respect to the places of $\mathcal{F}$ above $S$.\\[5mm] Denote the completion of $A$ with respect to $\mathfrak{p}$ by  $A_{\mathfrak{p}}$ and denote the Iwasawa algebra $A_{\mathfrak{p}}[\left[\operatorname{Gal}(\mathcal{F}/F)\right]]$ by $\Lambda(A_{\mathfrak{p}})$. We will equip a $\Lambda( A_{\mathfrak{p}})-$module on the fine Selmer group $\text{Sel}^S_0(\phi[\mathfrak{p}^{\infty}]/\mathcal{F})$ and there is an induced $\Lambda(A_{\mathfrak{p}})-$module structure on its Pontryagin dual  $Y^S(\phi[\mathfrak{p}^{\infty}]/\mathcal{F})$. We will show that its Pontryagin dual $Y^S(\phi[\mathfrak{p}^{\infty}]/\mathcal{F})$ is a finitely generated $\Lambda( A_{\mathfrak{p}})-$module such that its $\mu-$invariant vanishes. Furthermore, we show that the rank of $Y^S(\phi[\mathfrak{p}^{\infty}]/\mathcal{F})$ as an $ A_{\mathfrak{p}}-$module is equal to the $\lambda-$ invariant and we provide an upper bound for $\lambda$. \\[5mm] We summarize our main results in the following statement, which is a generalization of $\cite[\textbf{Theorem 1.1}]{bib1}$ in our context.
\begin{theorem} \label{theorem 1.3} Let $\phi$ be a Drinfeld module over $F$ of rank $r$ and $\mathfrak{p}$ be a non-zero prime ideal of $A$. Let $\mathcal{F}$ be the constant $\mathbb{Z}_p-$extension of $F$. We consider the set $S$ as in $(\ref{S})$ and denote the set of places in $\mathcal{F}$ above $S$ by $S(\mathcal{F})$, then the following statements hold
\begin{enumerate}[label={\arabic{enumi}}.]
\item The Pontryagin dual $Y^S(\phi[\mathfrak{p}^{\infty}]/\mathcal{F})$ of the fine Selmer group $\text{\normalfont\ Sel}_0^S(\phi[\mathfrak{p}^{\infty}]/\mathcal{F})$ is a finitely generated torsion $\Lambda( A_{\mathfrak{p}})-$module such that its $\mu-$Iwasawa invariant vanishes.
\item The Pontryagin dual $Y^S(\phi[\mathfrak{p}^{\infty}]/\mathcal{F})$ is a finitely generated module over $ A_{\mathfrak{p}}$ with its $A_{\mathfrak{p}}-$rank equal to the $\lambda-$Iwasawa invariant.
\item Let $\text{\normalfont \ Sel}_0^S(\phi[\mathfrak{p}]/\mathcal{F})$ be the residual fine Selmer group. Then the $\lambda-$Iwasawa invariant satisfies the following bound
\begin{align*}
    \lambda\leq {\dim}_{\mathbb{F}_\mathfrak{p}} \text{\normalfont\ Sel}_0^S(\phi[\mathfrak{p}]/\mathcal{F}) + \sum_{w \in S(\mathcal{F})} {\dim}_{\mathbb{F}_{\mathfrak{p}}}(H^0(\mathcal{F}_w^{\operatorname{sep}}/\mathcal{F}_w, \phi[\mathfrak{p^{\infty}}]\otimes_{ A_{\mathfrak{p}}}\mathbb{F}_{\mathfrak{p}}))
\end{align*} where $\mathbb{F}_{\mathfrak{p}}$ is the residue field $ A_{\mathfrak{p}}/\mathfrak{p}A_{\mathfrak{p}}$.
\end{enumerate}
\end{theorem} \noindent We apply the methods from $\cite{bib1}$ to prove our main results: we redefine the Pontryagin dual of a primary $\Lambda(A_{\mathfrak{p}})-$module. Since the Selmer group $\operatorname{Sel}_0^S(\phi[\mathfrak{p}^{\infty}]/\mathcal{F})$ is $\mathfrak{p}-$primary, we prove that the statement $(1)$ in $\textbf{Theorem }\ref{theorem 1.3}$ is equivalent to the finiteness of the $\mathfrak{p}-$torsion $\operatorname{Sel}_0^S(\phi[\mathfrak{p}^{\infty}]/\mathcal{F})[\mathfrak{p}]$ by studying the properties of Pontryagin duals (see \textbf{Theorem }\ref{theorem 4.6}). Furthermore, using $\emph{Snake Lemma}$, we prove that the finiteness of $\operatorname{Sel}_0^S(\phi[\mathfrak{p}^{\infty}]/\mathcal{F})[\mathfrak{p}]$ is equivalent to the finiteness of the residual fine Selmer group $\operatorname{Sel}_0^{S}(\phi[\mathfrak{p}]/\mathcal{F})$. \\[5mm]
To show the finiteness of the residual fine Selmer group $\operatorname{Sel}_0^{S}(\phi[\mathfrak{p}]/\mathcal{F})$, we need to provide a reformulation of Iwasawa theory of constant $\mathbb{Z}_p-$extension of global function fields: denote the maximal unramified abelian extension of $F_n$ in $F_n^{\operatorname{sep}}$ such that its degree over $F_n$ is a $p-$power and the places in $F_n$ above $S$ split completely by $H_p^{S}(F_n)$. We prove that
\begin{align*}
    X_S(\mathcal{F}) : = \varprojlim_n \operatorname{Gal}(H_p^S(F_n)/F_n)
\end{align*} is a finitely generated and torsion $\mathbb{Z}_p[\left[T\right]]-$module. Furthermore, using the theory of zeta function over global function fields (cf. \cite[\textbf{Chapter }11]{bib5}), we show that the $\mu-$invariant $X_S(\mathcal{F})$ vanishes and therefore $X_S(\mathcal{F})$ is pseudo-isomorphic to a free $\mathbb{Z}_p-$module of finite rank, which can be used to prove the finiteness of $\operatorname{Sel}_0^S(\phi[\mathfrak{p}]/\mathcal{F})$ (see \textbf{Proposition }\ref{proposition 5.5}).
\\[5mm] The statement (2) in $\textbf{Theorem }\ref{theorem 1.3}$ is an easy consequence of the statement $(1)$ and the statement $(3)$ is derived from some straightforward computations (see \textbf{Proposition }\ref{proposition 5.4}). 
\begin{center}
    \section{Theory of general Drinfeld modules}\label{sec2}
\end{center}
\justifying
In this section, we summarize the theory of Drinfeld modules.  We omit the proofs for some results and refer the readers to \cite[\textbf{Appendix A}]{bib2} for details. For the remainder of this section, the number $q$ is the power of a prime number $p$ and a global function field $F$ is always a finite extension of $\mathbb{F}_q(T)$ such that $\mathbb{F}_q$ is the full constant subfield of $F$. Furthermore, a place of a global function field $F$ refers to the maximal ideal of some valuation ring $\mathcal{O} \subsetneq F$ (cf. \cite[\textbf{1.1}]{bib8}). Denote the set of places of $F$ by $\Omega_F$ and the discrete valuation uniquely associated to a place $\mathfrak{p}$ in $\Omega_F$ by $v_{\mathfrak{p}}$. Given a fixed place $\infty$ of $F$, we consider the ring $A$ where
\begin{align} \label{A}
    A:= \{ x \in F : v_{\mathfrak{p}}(x) \geq 0, \forall \mathfrak{p} \in \Omega_F \setminus \{ \infty \} \}.
\end{align} and we consider the natural embedding $\gamma: A  \hookrightarrow F$. \\[5mm]
\begin{definition} Denote by $F\{\tau\}$ the ring of twisted polynomials (cf. \emph{\cite[\textbf{Definition 3.1.8}]{bib2}}). A Drinfeld module $\phi$ over $F$ is an $\mathbb{F}_q-$algebra homomorphism
\begin{align*}
    \phi: A & \longrightarrow F\{\tau\} \\
    a &\longmapsto \phi_a
\end{align*} such that $\phi(A) \nsubseteq F$ and $\partial\phi_{a} = \gamma(a)$ where $\partial$ is the formal derivative (cf. \emph{\cite[\textbf{Definition 4.4.1}]{bib13}}).
\end{definition} \noindent Our main interest with respect to a Drinfeld module $\phi$ over $F$ lies in studying its $\mathfrak{p}-$torsion where $\mathfrak{p}$ is a non-zero prime ideal of $A$ as in $(\ref{A})$. Given a non-zero ideal $I$ of $A$, there exists $\phi_I$ in $F\{ \tau \}$ such that the right ideal $F\{\tau\}\phi(I)$ is generated by $\phi_I$ (cf. \cite[\textbf{Corollary 3.1.15}]{bib2}). Furthermore, we fix an algebraic closure $\overline{F}$ of $F$.
\begin{definition} Let $I$ be a non-zero ideal of $A$. Then the $I-$torsion of $\phi$ is defined as
\begin{align*}
    \phi[I] := \{ x \in \overline{F} : \phi_I(x) = 0 \}.
\end{align*}
\end{definition}\noindent The following theorem is useful throughout this paper. Moreover, as an immediate consequence, the rank of a Drinfeld module $\phi$ over $F$ is always a positive integer.
\begin{theorem} Let $\phi$ be a Drinfeld mdule over $F$. Let $\mathfrak{p}$ be a non-zero prime ideal of $A$. Then for any $n$ in $\mathbb{N}_{>0}$, there is an isomorphism of $A-$modules
\begin{align*}
    \phi\left[\mathfrak{p}^n\right] \simeq \left(A/\mathfrak{p}^n\right)^{\oplus r}
\end{align*} where $r$ is the rank of $\phi$ \normalfont{(cf. \cite[\textbf{Definition A.6, Theorem A.12}]{bib2})}.
\end{theorem}
\begin{proof} See \cite[\textbf{Theorem A.12 (1)}]{bib2}.
\end{proof} \noindent For each $n$ in $\mathbb{N}_{>0}$ and a non-zero prime ideal $\mathfrak{p}$ of $A$, there is an isomorphism of rings
\begin{align*}
    A/\mathfrak{p}^nA \simeq A_{\mathfrak{p}}/\mathfrak{p}^nA_{\mathfrak{p}}
\end{align*} where $ A_{\mathfrak{p}}$ is the completion of $A$ with respect to $\mathfrak{p}$ and so it is a discrete valuation ring. For a fixed uniformizer $\pi$ of $A_{\mathfrak{p}}$, if we consider the injection
\begin{align} \label{map1}
   \iota_n: A_{\mathfrak{p}}/\mathfrak{p}^nA_{\mathfrak{p}} & \longrightarrow A_{\mathfrak{p}}/\mathfrak{p}^{n+1}A_{\mathfrak{p}},\\
    [a] & \longmapsto [\pi a] \label{map2}
\end{align} for each $n$, then we may define
\begin{align*}
    \phi[\mathfrak{p}^{\infty}] := \varinjlim_n \phi[\mathfrak{p}^n]
\end{align*} where each map inside the direct system is the natural extension of $\iota_n$ defined in ($\ref{map1}$), ($\ref{map2}$) to the corresponding $r-$powers. Furthermore, since each $\iota_n$ is an injection, we may identify 
\begin{align*}
    \phi[\mathfrak{p}^{\infty}]  = \bigcup_{n\geq 1}  \phi[\mathfrak{p}^n].
\end{align*} Moreover, we have an isomorphism of $A-$modules
\begin{align} \label{rankr isomorphism}
    \phi[\mathfrak{p}^{\infty}] \simeq \left(F_{\mathfrak{p}}/A_\mathfrak{p}\right)^{\oplus r}
\end{align} where $F_{\mathfrak{p}}$ is the field of fractions of $A_{\mathfrak{p}}$. This isomorphism can be established in the same way as $\mathbb{Q}_p/\mathbb{Z}_p \simeq \varinjlim_n \mathbb{Z}/p^n\mathbb{Z}$ and we leave this to the readers to verify. Finally, since there exists $\phi_{\mathfrak{p}}$ in $F\{ \tau \}$ such that
\begin{align*}
    \phi[\mathfrak{p}] = \{ \alpha \in \overline{F} : \phi_{\mathfrak{p}}(\alpha) = 0 \},
\end{align*} we consider the surjective maps
\begin{align*}
    \forall n \in \mathbb{N},\quad\phi_{\pi}^{n+1, n}: \phi[\mathfrak{p}^{n+1}] &\longrightarrow \phi[\mathfrak{p}^{n}] \\
    \alpha&\longmapsto \phi_{\mathfrak{p}}(\alpha)
\end{align*} with $\operatorname{ker}(\phi_{\pi}^{n+1, n}) = \phi[\mathfrak{p}]$ for each $n$. Therefore, we have the following short exact sequence of $A_{\mathfrak{p}}-$modules by taking the direct limit
\begin{align} \label{galoiscohomology}
    0 \rightarrow \phi[\mathfrak{p}] \rightarrow  \phi[\mathfrak{p}^{\infty}] \rightarrow \phi[\mathfrak{p}^{\infty}] \rightarrow 0.
\end{align} 
 Given a Drinfeld module $\phi$ over $F$ and a fixed non-zero prime ideal $\mathfrak{p}$ of $A$, we ease the notation by denoting the place uniquely associated to $\mathfrak{p}$ also by $\mathfrak{p}$. We need to show that the set $S$ as in $(\ref{S})$ is finite, and it suffices to show that any Drinfeld module $\phi$ over $F$ has only finitely many bad reductions. 
\begin{definition} Let $\phi$ be a Drinfeld module over $F$ and  $\mathfrak{p}$ be some non-zero prime ideal of $A$. We say $\phi$ has a stable reduction at $\mathfrak{p}$ if $\phi$ is isomorphic to some Drinfeld module $\psi$
\begin{align*}
    \psi: A \longrightarrow A_{\mathfrak{p}} \{ \tau \}
\end{align*} such that 
\begin{align*}
    \overline{\psi}: A \xrightarrow{\psi} A_{\mathfrak{p}}\{\tau\} \twoheadrightarrow\mathbb{F}_{\mathfrak{p}}\{ \tau \}
\end{align*} is a Drinfeld module over $\mathbb{F}_{\mathfrak{p}}$ where $\mathbb{F}_{\mathfrak{p}} = A_{\mathfrak{p}}/\mathfrak{p}A_{\mathfrak{p}}$.
Furthermore, we say $\phi$ has a good reduction at $\mathfrak{p}$ if it has a stable reduction at $\mathfrak{p}$ and the rank of the induced Drinfeld module $\overline{\psi}$ over $\mathbb{F}_{\mathfrak{p}}$ is equal to the rank of the Drinfeld module $\phi$ over $F$. We say $\mathfrak{p}$ is a bad reduction of $\phi$ if $\phi$ does not have a good reduction at $\mathfrak{p}$.
\end{definition}
\begin{remark} Let $\phi$ be a Drinfeld module over $F$. Let $\mathfrak{X} = \{ T_i \}_{i=1}^n$ be a finite subset of $A$ generating $A$ as an $\mathbb{F}_q-$algebra. We may assume without loss of generality that the Drinfeld module $\phi$ such that 
\begin{align*}
   \forall T_i \in \mathfrak{X},\quad \phi_{T_i} = T_i \tau^0 + a_{i1} \tau^1 + \cdots + a_{ir_i}\tau^{r_i}
\end{align*} and all its coefficients are in $A$. Therefore, we see that $\phi$ has a bad reduction at a non-zero prime ideal $\mathfrak{p}$ if and only if $a_{ir_i}$ is contained in $\mathfrak{p}$ for any $1\leq i \leq n$. Hence, the set of bad reductions of $\phi$ is finite.
\end{remark}
\begin{center}
    \section{Iwasawa theory of constant $\mathbb{Z}_p-$extension}
\end{center}
\justifying
 In this section, we study the Iwasawa theory of constant $\mathbb{Z}_p-$extension of global function fields. 
Given $F$ a global function field over $\mathbb{F}_q$, let us take a tower of extensions
\begin{align} \label{constant extension}
    F = F_0 \subseteq \cdots \subseteq F_n \subseteq F_{n+1} \subseteq \cdots \subseteq \mathcal{F}
\end{align} where each $F_n$ is defined as the compositum $F\mathbb{F}_{q^{p^n}}$. In other words, each $F_n$ is a constant extension of $F$ such that $[F_n:F] = p^n$. We call $\mathcal{F}$ in $(\ref{constant extension})$ the constant $\mathbb{Z}_p-$extension of $F$. For a fixed $n$ and some fixed place $\mathfrak{p}$ in $F$, we denote the unique valuation associated to a place $\mathfrak{P}_n$ of $F_n$ above $\mathfrak{p}$ by $v_{{\mathfrak{P}_n}}$ and we consider the following valuation ring with respect to $\mathfrak{P}_n$ which contains $\mathfrak{P}_n$ as its unique maximal ideal
\begin{align*}
     \mathcal{O}_{\mathfrak{P}_n} : = \{ x \in F_n: v_{\mathfrak{P}_n}(x) \geq 0 \}.
\end{align*}
\begin{definition} Let $F$ be a global function field over $\mathbb{F}_q$. For each $F_n$ in $(\ref{constant extension})$, the degree of a place $\mathfrak{P}_n$ of $F_n$ is defined to be
\begin{align*}
    \operatorname{deg } \mathfrak{P}_n : = [\mathcal{O}_{\mathfrak{P}_n}/\mathfrak{P}_n : \mathbb{F}_{q^{p^n}}].
\end{align*}
\end{definition}
\begin{proposition} \label{proposition 3.2} Let $S$ be a finite subset of places of $F$. Let $H_p^{S}(F_n)$ be the maximal unramified abelian extension of $F_n$ in $F_n^{\operatorname{sep}}$ such that its degree over $F_n$ is a $p-$power and the places of $F_n$ above $S$ split completely. Then for sufficiently large $n$, we have
\begin{align*}
    \forall m > n,\quad H_p^{S}(F_n) \cap F_m = F_n.
\end{align*}
\end{proposition}    
\begin{proof} For each $n \geq 0$, we denote the set of places of $F_n$ above $S$ by $S(F_n)$. Furthermore, we denote the maximal unramified abelian extension of $F_n$ in $F_n^{\operatorname{sep}}$ such that the places of $F_n$ above $S$ split completely by $H^S(F_n)$.  As a direct result of $\cite[\textbf{Theorem 1.3}]{bib14}$, the full constant subfield of $H^S(F_n)$ is $\mathbb{F}_{q^{\delta_n p^n}}$ where
\begin{align} \label{delta}
    \delta_n = \operatorname{gcd}_{\mathfrak{P} \in S(F_n)}\{ \operatorname{deg} \mathfrak{P} \}.
\end{align} \textbf{Claim:} For all $n \geq 0$, we have $\delta_{n+1} \leq \delta_n$ and there exists some $N \in \mathbb{N}$ such that $\delta_m = \delta_N$ for all $m \geq N$. Furthermore, we have that $p \nmid \delta_N$. \\[5mm]
\textbf{Proof of claim:} The first part of the claim is a direct consequence of $\cite[\textbf{Theorem 3.6.3 (c)}]{bib8}$. For any place $\mathfrak{P}_n$ of $F_n$ above some place $\mathfrak{p}$ of $F$, the following equality holds
\begin{align} \label{degp}
    \operatorname{deg} \mathfrak{P}_n = \frac{\operatorname{deg} \mathfrak{p}}{\operatorname{gcd}(\operatorname{deg} \mathfrak{p}, p^n)}
\end{align} as a result of $\cite[\textbf{Lemma 5.1.9(d)}]{bib8}$. Therefore, for sufficiently large $n$, we derive that
\begin{align*}
    \operatorname{gcd}(\operatorname{deg} \mathfrak{P}_n, p^n) = p^{v_p(\operatorname{deg} \mathfrak{P}_n)}.
\end{align*} Hence, we must have that $p \nmid \operatorname{deg} \mathfrak{P}_n$ $(\ref{degp})$ and therefore, the prime number $p$ does not divide $\delta_n$ $(\ref{delta})$ for sufficiently large $n$, which equals to $\delta_N$ for some fixed $N \in \mathbb{N}$ by the first part of the claim. \hspace{25mm} \checkmark
\\[5mm]
Now for sufficiently large $n$ and $m > n$, we have that the intersection of the full constant subfields of $H^S(F_n)$ and $F_m$ is
\begin{align*}
  \mathbb{F}_{q^{\delta_N p^n}} \cap \mathbb{F}_{q^{p^m}} = \mathbb{F}_{q^{\operatorname{gcd}(\delta_N p^n, p^m)}} = \mathbb{F}_{q^{p^n}}
\end{align*} where the last equality is a direct result of the second part of the claim. 
We observe that for sufficiently large $n$ and $m > n$, the extension $H^S(F_n) \cap F_m$ is a subfield of $F_m$ with $\mathbb{F}_{q^{p^n}}$ as its full constant subfield and so we must have
\begin{align*}
   \forall n \gg 0, m > n, \quad H^S(F_n) \cap F_m = F_n.
\end{align*} Lastly, we derive that
\begin{align*}
    \forall n \gg 0, m>n, \quad F_n \subseteq H_p^S(F_n) \cap F_m \subseteq H^S(F_n) \cap F_m = F_n,
\end{align*} which yields our desired result.
\end{proof}\noindent  As a direct result of \textbf{Proposition }\ref{proposition 3.2}, the following equalities hold
\begin{align} \label{fact}
 \forall n \gg 0,\quad    H_p^S(F_n) \cap \mathcal{F} = F_n.
\end{align}  If we set
\begin{align} \label{X_n}
 \forall n \geq 0, \quad X_n := \operatorname{Gal}(H_p^{S}(F_n)/F_n),
\end{align}
we compute that
\begin{align*}
    \varprojlim_n X_n &= \varprojlim \operatorname{Gal}(H_p^{S}(F_n)/F_n)\\
       & \stackrel{(\ref{fact})} = \varprojlim_n \text{Gal}(H_p^S(F_n) /H_p^S(F_n) \cap \mathcal{F}) \\
    & \simeq \varprojlim_n \text{Gal}(H_p^S(F_n) \cdot \mathcal{F} / \mathcal{F}) \\
    & \simeq \text{Gal}(H_p^S(\mathcal{F})/ \mathcal{F}).
\end{align*} We denote the inverse limit above by $X_S(\mathcal{F})$ and denote the Iwasawa algebra $\mathbb{Z}_p\left[\left[\operatorname{Gal}(\mathcal{F}/F)\right]\right]$ by $\Lambda(\mathbb{Z}_p)$. The inverse limit $X_S(\mathcal{F})$ is a module over $\Lambda(\mathbb{Z}_p)$ induced by the action
\begin{align*}
    \forall n \geq 0, \quad \operatorname{Gal}(F_n/F) \times \operatorname{Gal}(H_p^S(F_n)/F_n) & \longrightarrow \operatorname{Gal}(H_p^S(F_n)/F_n), \\
    (\gamma, x) & \longmapsto \Tilde{\gamma} x \Tilde{\gamma}^{-1}.
\end{align*} where $\Tilde{\gamma}$ is an extension of $\gamma$ from $\operatorname{Gal}(H_p^{S}(F_n)/F)$. Since each $\operatorname{Gal}(H_p^S(F_n)/F_n)$ is abelian, this action is well-defined (cf. \cite[p.278]{bib3}). Furthermore, we will see that the usual statements of Iwasawa theory of number fields hold for $X_S(\mathcal{F})$. 
\begin{lemma} \label{lemma 3.3} Let $\mathfrak{p}$ be a place of $F$. Let $D_{\mathfrak{p}}(\mathcal{F}/F)$ be the decomposition group of $\mathfrak{p}$ with respect to the constant $\mathbb{Z}_p-$extension $\mathcal{F}$ of $F$. Then $D_{\mathfrak{p}}(\mathcal{F}/F)$ is non-trivial.
\end{lemma}
\begin{proof} Given a place $\mathfrak{p}$ of $F$, we fix a place $\mathfrak{P}$ of $\mathcal{F}$ above $\mathfrak{p}$. Since the constant extension is abelian and unramified, the decomposition group is independent of the choice of the place $\mathfrak{P}$ above $\mathfrak{p}$ and there is an isomorphism 
\begin{align*}
    D_{\mathfrak{p}}(\mathcal{F}/F) \simeq \operatorname{Gal}(\mathcal{F}_\mathfrak{P}/F_{\mathfrak{p}})
\end{align*} where $\mathcal{F}_{\mathfrak{P}}$ and $F_{\mathfrak{p}}$ are the residue fields of $\mathcal{F}$ at $\mathfrak{P}$ and $F$ at $\mathfrak{p}$, respectively  (cf. \cite[\textbf{Definition 1.1.14}]{bib8}, \cite[\textbf{Theorem 3.8.2(c)}]{bib8}). However, the Galois group $\operatorname{Gal}(\mathcal{F}_\mathfrak{P}/F_{\mathfrak{p}})$ is the kernel of the following restriction map
\begin{align*}
   \varphi: \operatorname{Gal}(\mathcal{F}_{\mathfrak{P}}/\mathbb{F}_q) \longrightarrow \operatorname{Gal}(F_\mathfrak{p}/\mathbb{F}_q)
\end{align*} where the domain is infinite and the codomain is finite. Therefore, the kernel of $\varphi$ is infinite and so the decomposition group $D_{\mathfrak{p}}(\mathcal{F}/F)$ cannot be trivial. 
\end{proof}
\begin{proposition} \label{proposition 3.4}  Let $S$ be a finite subset of places of $F$. Then there exists some $n \geq 0$ such that every place in $S$ is totally inert in the extension $\mathcal{F}/F_n$.
\end{proposition}  
\begin{proof} The proof is similar to $\cite[\textbf{Lemma 13.3}]{bib3}$ but we repeat it here. By $\textbf{Lemma }\ref{lemma 3.3}$, the decomposition group $D_{\mathfrak{p}}(\mathcal{F}/F)$ of any place $\mathfrak{p}$ is a non-trivial subgroup of $\operatorname{Gal}(\mathcal{F}/F)$ which is $\mathbb{Z}_p$ by construction. It follows that
\begin{align*}
   \exists n \geq 0, \quad  \bigcap_{\mathfrak{p} \in S} D_{\mathfrak{p}}(\mathcal{F}/F) = p^n\mathbb{Z}_p.
\end{align*} Therefore, we have that
\begin{align*}
    \operatorname{Gal}(\mathcal{F}/F_n) =  \bigcap_{\mathfrak{p} \in S} D_{\mathfrak{p}}(\mathcal{F}/F),
\end{align*} which implies that
\begin{align*}
    \forall \mathfrak{p} \in S, \quad \operatorname{Gal}(\mathcal{F}/F_n) \subseteq D_{\mathfrak{p}}(\mathcal{F}/F).
\end{align*} We further derive that
\begin{align*}
    \forall \mathfrak{p} \in S,\quad D_{\mathfrak{p}}(\mathcal{F}/F_n) = \operatorname{Gal}(\mathcal{F}/F_n) \cap D_{\mathfrak{p}}(\mathcal{F}/F) = \operatorname{Gal}(\mathcal{F}/F_n),
\end{align*}which means $F/F_n$ is totally inert for all $\mathfrak{p}$ in $S$.
\end{proof}
\noindent We replace the field $F_n$ in $\textbf{Proposition }\ref{proposition 3.4}$ with $F$ so the place $\mathfrak{p}$ is totally inert with respect to the extension $\mathcal{F}/F$ for each place $\mathfrak{p}$ in $S$. Denote $\operatorname{Gal}(H_p^{S}(\mathcal{F})/F)$ by $G$ and denote the decomposition group with respect to the extension $H_p^{S}(\mathcal{F})/F$ of $\mathfrak{p}$ by $D_{\mathfrak{p}}$. Since  any place $\mathfrak{P}$ in $\mathcal{F}$ above arbitrary place $\mathfrak{p}$ in $S$ splits completely in $H_p^S(\mathcal{F})$, we have that
\begin{align*}
    \forall \mathfrak{p} \in S,\quad D_{\mathfrak{p}} \cap X_S(\mathcal{F}) = 1.
\end{align*} Furthermore, since $\mathcal{F}/F$ is totally inert at any $\mathfrak{p}$ in $S$, we have
\begin{align*}
    \forall \mathfrak{p} \in S,\quad D_{\mathfrak{p}} \hookrightarrow G/X_S(\mathcal{F}) = \operatorname{Gal}(\mathcal{F}/F)
\end{align*} is surjective and hence bijective. Furthermore, we obtain that
\begin{align} \label{G}
    \forall \mathfrak{p} \in S,\quad G = D_{\mathfrak{p}}X_S(\mathcal{F}) = X_S(\mathcal{F})D_{\mathfrak{p}}.
\end{align} Moreover, we fix a place $\infty$ in $S$ and we identify $\operatorname{Gal}(\mathcal{F}/F)$ with $D_{\infty}$. Since
\begin{align*}
    \forall \mathfrak{p} \in S,\quad D_{\mathfrak{p}} \subseteq X_S(\mathcal{F})D_{\infty}
\end{align*} we have 
\begin{align*}
  \forall \mathfrak{p} \in S,\quad  \sigma_{\mathfrak{p}} = a_{\mathfrak{p}} \sigma_{\infty}
\end{align*} where $\sigma_{\mathfrak{p}}$ is a topological generator of $D_{\mathfrak{p}}$ (cf. \cite[p.279]{bib3}). Upon identification of $D_{\infty}$ and $\operatorname{Gal}(\mathcal{F}/F)$, there is a natural action of $\operatorname{Gal}(\mathcal{F}/F)$ on $X_S(\mathcal{F})$
\begin{align} 
    \operatorname{Gal}(\mathcal{F}/F) \times X_S(\mathcal{F}) &\longrightarrow X_S(\mathcal{F}), \label{Z_p structure}\\
    (g, x) & \longmapsto gxg^{-1}.\label{Z_p structure1}
\end{align} We denote this action by $x^g$ for $g$ in $\operatorname{Gal}(\mathcal{F}/F)$ acting on $x$ in $X_S(\mathcal{F})$.
\begin{lemma} With the notations and the replacement above, let $G^{\prime}$ be the closure of the commutator subgroup of $G$. Then we have
\begin{align*}
    G^{\prime} = X_S(\mathcal{F})^{\sigma_{\infty} - 1}.
\end{align*}
\end{lemma}
\begin{proof} See \cite[\textbf{Lemma 13.14}]{bib3}.
\end{proof}
\begin{lemma} \label{lemma 3.6} With the notations and the replacement above, let $Y_0$ be the $\mathbb{Z}_p-$submodule of $X_S(\mathcal{F})$ \normalfont{((\ref{Z_p structure}), (\ref{Z_p structure1}))} generated by $\{ a_\mathfrak{p} : \mathfrak{p} \in S \setminus \{\infty \} \}$ and $G^{\prime}$. Let $Y_n = \nu_n Y_0$ where
\begin{align*}
    v_n = 1 + \sigma_{\infty} + \sigma_{\infty}^2 + \cdots + \sigma_{\infty}^{p^n-1}.
\end{align*}Then 
\begin{align*}
    \forall n\geq 0,\quad X_n = \operatorname{Gal}(H_p^{S}(F_n)/F_n) \simeq \frac{X_S(\mathcal{F})}{Y_n}.
\end{align*}
\end{lemma}
\begin{proof} If $n = 0$, we have $F \subseteq H_p^{S}(F) \subseteq H_p^S(\mathcal{F})$. We claim that 
\begin{align} \label{Galois}
    \operatorname{Gal}(H_p^{S}(\mathcal{F})/H_p^{S}(F)) \simeq \overline{\left<G^{\prime},D_\mathfrak{p}, \mathfrak{p} \in S\right>}.
\end{align} Indeed, the field extension $ H_p^{S}(\mathcal{F})^{\overline{\left<G^{\prime},D_\mathfrak{p}, \mathfrak{p} \in S\right>}}/F$ is abelian, unramified. Furthermore, every place in $S$ splits completely in this extension. Therefore, we derive that 
\begin{align*}
    H_p^{S}(\mathcal{F})^{\overline{\left<G^{\prime},D_\mathfrak{p}, \mathfrak{p} \in S\right>}} = H_p^{S}(F).
\end{align*} by definition of $H_p^{S}(F)$. Now, the isomorphism $(\ref{Galois})$ is a result of fundamental theorem of Galois theory. Furthermore, we observe that
\begin{align*}
    \overline{\left<G^{\prime}, D_\mathfrak{p}, \mathfrak{p} \in S\right>} =  \overline{\left<G^{\prime},a_\mathfrak{p}, \mathfrak{p} \in S \setminus \{ \infty\}\right>}D_{\infty},
\end{align*} and we compute that
\begin{align*}
    X_S(\mathcal{F})/Y_{0} & = X_S(\mathcal{F})D_{\infty} /Y_{0}D_{\infty}  \\
    & = \operatorname{Gal}(H_{p}^{S}(\mathcal{F})/F)/Y_{0}D_{\infty}  \\
    & =  \operatorname{Gal}(H_{p}^{S}(\mathcal{F})/F)/\overline{\left<G^{\prime},D_\mathfrak{p}, \mathfrak{p} \in S\right>} \\
    & = \operatorname{Gal}(H_p^{S}(F)/F) \\
    & = X_0.
\end{align*} For $n \geq 1,$ we have
\begin{align*}
    \forall \mathfrak{p} \in S, \quad \sigma_\mathfrak{p}^{p^n} = (\nu_na_\mathfrak{p}) (\sigma_{\infty})^{p^n}
\end{align*} and
\begin{align*}
    X_S(\mathcal{F})^{\sigma_{\infty}^{p^n}-1}=X_S(\mathcal{F})^{\nu_n(\sigma_{\infty} - 1)} =  (G^{\prime})^{\nu_n}.
\end{align*} Similarly to $(\ref{G})$, we derive that
\begin{align*}
    \forall n \geq 1, \quad \operatorname{Gal}(H_p^{S}(\mathcal{F})/F_n) =X_S(\mathcal{F}) D_{\infty}(H_p^{S}(\mathcal{F})/F_n)
\end{align*} and 
\begin{align*}
    \forall n \geq 1, \quad \operatorname{Gal}(H_p^{S}(\mathcal{F})/H_p^{S}(F_n)) = Y_{n}D_{\infty}(H_p^{S}(\mathcal{F})/F_n).
\end{align*}
Similarly to the case $n = 0$, we deduce that $X_S(\mathcal{F})/Y_n$ is $X_n$ for all $n \geq 1$ and thus conclude the proof.
\end{proof} \noindent The Iwasawa algebra $\Lambda(\mathbb{Z}_p)$ is isomorphic to $\mathbb{Z}_p\left[\left[T\right]\right]$ (cf. \cite[\textbf{Theorem 7.1}]{bib3}). We give one of the main results of this section.
\begin{theorem} \label{theorem 3.7} Let $F$ be any global function field. Let $\mathcal{F}/F$ be the constant $\mathbb{Z}_p-$extension. Let $S$ be a finite set of places of $F$. Then $X_S(\mathcal{F})$ is a finitely generated and torsion $\Lambda(\mathbb{Z}_p)-$module, i.e., it is pseudo-isomorphic to
\begin{align} \label{pseudo}
      \left(\bigoplus_{i=1}^s \frac{\mathbb{Z}_p\left[\left[T\right]\right]}{(p^{\mu_i})}\right) \oplus \left(\bigoplus_{j=1}^t \frac{\mathbb{Z}_p\left[\left[T\right]\right]}{(f_j(T))}\right).
\end{align}  where each $f_j(T)$ is a distinguished polynomial \emph{(cf. \cite[p.115, l-15]{bib3})}. Furthermore, the following equalities 
\begin{align*}
   \forall n\gg 0,  e_n = \lambda n + \mu p^n + \nu
\end{align*} hold where $e_n$ is the $p-$order of the cardinality of $X_n$ and $\lambda, \mu, \nu$ are constant integers such that $\lambda, \mu$ are Iwasawa invariants of $X_S(\mathcal{F})$ from $(\ref{pseudo})$ \normalfont{(cf. \cite[p.10]{bib1})}.
\end{theorem}
\begin{proof} There exists some $n \geq 0$ such that every place $\mathfrak{p}$ in $S$ is totally inert in $\mathcal{F}/F_n$ as a result of $\textbf{Proposition }\ref{proposition 3.4}$. For each $m \geq n$, we consider
\begin{align*}
  v_{m,n} := \frac{v_m}{v_n} = 1 + \sigma_{\infty}^{p^n} + \cdots + \sigma_{\infty}^{p^m-p^n}.
\end{align*} As a result of $\textbf{Lemma }\ref{lemma 3.6}$, we obtain that
\begin{align*}
    \forall m\geq n,\quad X_m \simeq  \frac{X_S(\mathcal{F})}{Y_m}
\end{align*} where $Y_m = v_{m,n}Y_n$. Now, the remainder of the proof appears in \cite[p281-285]{bib3}.
\end{proof} \noindent We wish to further see that the $\mu-$invariant of $X_S(\mathcal{F})$ vanishes. This requires us to interpret $X_S(\mathcal{F})$ as an inverse limit of class groups. Hence, we give the following definition.
\begin{definition} Let $F$ be a global function field. Let $\mathcal{O}_S(F)$ be the ring of $S-$integer in $F$ where $S$ is a finite set consisting of places in $F$. Denote the group of Weil divisors on $\emph{Spec}(\mathcal{O}_S(F))$ of degree 0 by $\emph{Div}^0(\mathcal{O}_S(F))$ and its subgroup consisting of all principal divisors on $\emph{Spec}(\mathcal{O}_S(F))$ by $\emph{Princ}(\mathcal{O}_S(F))$. Then the class group on $\emph{Spec}(\mathcal{O}_S(F))$ is the quotient group
\begin{align*}
  \emph{Cl}^S(F) := \frac{\emph{Div}^0(\mathcal{O}_S(F))}{\emph{Princ}(\mathcal{O}_S(F))}.
\end{align*}
\end{definition}
\begin{remark} \label{remark 3.9} Let $F$ be a global function field. Denote the class group of $F$ by $\operatorname{Cl}(F)$ \normalfont{(cf. \cite[\textbf{Definition 5.1.2}]{bib8})}. Then $\operatorname{Cl}^S(F)$ is a quotient of $\operatorname{Cl}(F)$ \normalfont{(cf. \cite[\textbf{II, Proposition 6.5}]{bib9})}.
\end{remark}
\begin{proposition} \label{proposition 3.10} Let $F$ be a global function field. Let $S$ be a finite set of places in $F$. Then there is an isomorphism of groups
\begin{align*}
     \emph{Gal}(H^S(F)/F) \simeq  \emph{Cl}^S(F)
\end{align*} where $H^S(F)$ is the maximal abelian extension of $F$ contained in $F^{\emph{sep}}$ such that $H^S(F)$ is unramified at all places of $F$ and all places in $S$ split completely.
\end{proposition}
\begin{proof} See \cite[p.64, l-3]{bib10}.
\end{proof}
\begin{remark} \label{remark 3.11} Let $S$ be a finite set of places of $F$. To simply notations, we denote the class group on $\operatorname{Spec}(\mathcal{O}_{S(F_n)}(F_n))$ by $\operatorname{Cl}^S(F_n)$. Furthermore, we denote the $p-$Sylow subgroup of $\emph{Cl}^S(F_n)$ by $\operatorname{Cl}_p^{S}(F_n)$. As a direct result of $\normalfont{\textbf{Proposition }}\ref{proposition 3.10}$, there is an isomorphism
\begin{align*}
    X_n = \operatorname{Gal}(H_p^{S}(F_n)/F_n) \simeq \operatorname{Cl}_p^{S}(F_n).
\end{align*}
\end{remark}
\begin{theorem} \label{theorem 3.12} With respect to the notations above, there is a pseudo-isomorphism $\varphi$ of $\Lambda(\mathbb{Z}_p)-$modules
\begin{align*}
   \varphi: X_S(\mathcal{F}) \longrightarrow \mathbb{Z}_p^{\oplus\lambda}
\end{align*} where $\lambda$ is the $\lambda-$Iwasawa invariant of $X_S(\mathcal{F})$. Furthermore, there is an isomorphism of $\Lambda(\mathbb{Z}_p)-$modules
\begin{align*}
    X_S(\mathcal{F}) \simeq \mathbb{Z}_p^{\oplus \lambda} \oplus \operatorname{ker}(\varphi)
\end{align*} and therefore, $X_S(\mathcal{F})/pX_{S}(\mathcal{F})$ is finite.
\end{theorem}
\begin{proof} As a direct application of \cite[\textbf{Theorem 11.5}]{bib5}, we know that
\begin{align*}
   \forall n \geq 0,\quad  e_n^{\prime} = \lambda^{\prime}n + \nu^{\prime}
\end{align*} where $e_n^{\prime}$ is the $p-$order of the cardinality of $\operatorname{Cl}(F_n)$. 
Combining $\textbf{Remark }\ref{remark 3.9}$ and $\textbf{Remark }\ref{remark 3.11}$, we obtain that
\begin{align} \label{inequality}
    \forall n \geq 0, \quad e_n \leq e_{n}^{\prime}
\end{align} where $e_n$ is the $p-$order of the cardinality of $X_n$ ($\ref{X_n}$). As a result of $\textbf{Theorem } \ref{theorem 3.7}$, we may rewrite the inequality $(\ref{inequality})$
\begin{align*}
    \forall n \gg 0,\quad e_n = \lambda n + \mu p^n + \nu \leq e_n^{\prime} = \lambda^{\prime} n + \nu^{\prime}
\end{align*} We further observe that the $\mu-$invariant of $X_S(\mathcal{F})$ must vanish. Hence, we conclude the following exact sequence of $\Lambda(\mathbb{Z}_p)-$modules
\begin{align*}
    0 \rightarrow \text{ker}(\varphi) \hookrightarrow X_S(\mathcal{F})  \xrightarrow{\varphi} \mathbb{Z}_p^{\oplus \lambda} \twoheadrightarrow \text{coker}(\varphi) \rightarrow 0
\end{align*} where $\varphi$ is a pseudo-isomorphism. Since $\mathbb{Z}_p$ is a P.I.D and any submodule of a free module over P.I.D is free, we must have that $\text{coker}(\varphi)$ is 0 and we further deduce
\begin{align*}
    X_S(\mathcal{F}) \simeq \mathbb{Z}_p^{\oplus \lambda} \oplus \text{ker}(\varphi).
\end{align*} because $\mathbb{Z}_p^{\oplus \lambda}$ is projective. Thus, the quotient in the statement is indeed finite.
\end{proof}
\begin{center}
    \section{Pontryagin dual}\label{sec3}
\end{center}
\justifying
In this section, we give the definition of the Pontryagin dual of a $\mathfrak{p}-$primary Iwasawa module in our context and we study its properties. We are interested in the equivalent condition for the Pontryagin dual of a $\Lambda( A_{\mathfrak{p}})-$module to be torsion and finitely generated with vanishing $\mu-$invariant, as in the statement of $\textbf{Theorem }\ref{theorem 1.3}$. 
\\[5mm]
\begin{definition} A $\Lambda( A_{\mathfrak{p}})-$module $M$ is $\mathfrak{p}-$primary if $ M =\bigcup_{n\geq 1} M[\mathfrak{p}^n].$
\end{definition}
\begin{definition} Denote the field of fraction of $A_{\mathfrak{p}}$ by $F_{\mathfrak{p}}$. The Pontryagin dual of a $\mathfrak{p}-$primary module $M$ is
\begin{align*}
    M^{\vee} : = \emph{Hom}_{A_{\mathfrak{p}}}(M, F_{\mathfrak{p}}/A_{\mathfrak{p}}).
\end{align*}
\end{definition}
\begin{remark} Denote the completion of $F$ with respect to $\infty$ by $F_{\infty}$. The ring $A$ as in \emph{(\ref{A})} is discrete and cocompact in $F_{\infty}$ \normalfont{(cf. \cite[\textbf{Lemma 7.6.16}]{bib2})}. Assuming that $M$ is $\mathfrak{p}-$primary, there is an isomorphism
\begin{align*}
    \operatorname{Hom}(M, F_{\infty}/A) \simeq \operatorname{Hom}_{A_{\mathfrak{p}}}(M, F_{\mathfrak{p}}/A_{\mathfrak{p}})
\end{align*} where the left hand side above is the usual Pontryagin dual.
\end{remark}
\noindent The following two lemmas will be needed to complete the proof of $\textbf{Theorem }\ref{theorem 4.6}$.
\begin{lemma} \label{Lemma 4.4} Denote $M^{\vee}$ by $N$. Then for all $n\geq 1$, there is an $A_{\mathfrak{p}}-$module isomorphism $(M[\mathfrak{p}^n])^{\vee} \simeq N/\mathfrak{p}^n N$.
\end{lemma}
\begin{proof} The quotient $F_{\mathfrak{p}}/A_{\mathfrak{p}}$ is a divisible $A_{\mathfrak{p}}-$module and therefore it is injective. This further leads to that the functor $\text{Hom}_{A_{\mathfrak{p}}}(\cdot, F_{\mathfrak{p}}/A_{\mathfrak{p}})$ is exact. Hence, we only need to show that the kernel of the following map is $\mathfrak{p}^nN$ 
\begin{align*}
    \varphi_n:  \text{Hom}_{A_{\mathfrak{p}}}(M, F_{\mathfrak{p}}/A_{\mathfrak{p}}) & \longrightarrow  \text{Hom}_{A_{\mathfrak{p}}}(M[\mathfrak{p}^n], F_{\mathfrak{p}}/A_{\mathfrak{p}}), \\
    f: M  \longrightarrow F_{\mathfrak{p}}/A_{\mathfrak{p}} & \longmapsto   f^{\prime}: M[\mathfrak{p}^n] \hookrightarrow M  \xrightarrow{f} F_{\mathfrak{p}}/A_{\mathfrak{p}}.
\end{align*} The inclusion $\mathfrak{p}^nN \subseteq \text{ker} (\varphi_n)$ is obvious. For the other inclusion, suppose $f$ is in $\text{ker}(\varphi_n)$ and $\pi$ is a uniformizer for $A_{\mathfrak{p}}$, we check that
\begin{align*}
    g: M &\longrightarrow F_{\mathfrak{p}}/A_{\mathfrak{p}}, \\
    m & \longmapsto \left[\frac{1}{\pi^n}\right]\cdot f(m),
\end{align*} is an $A_{\mathfrak{p}}-$module homomorphism. Therefore, we have that $f = \pi^ng$ belonging to $\mathfrak{p}^nN$.
\end{proof}
\begin{lemma}  \label{Lemma 4.5} Let $N$ and $N^{\prime}$ be two $\Lambda(A_{\mathfrak{p}})-$modules such that there is a pseudo-isomorphism $\varphi: N \longrightarrow N^{\prime}$. Then the induced homomorphism $\psi: N/\mathfrak{p}N\longrightarrow N^{\prime}/\mathfrak{p}N^{\prime}$ is a pseudo-isomorphism.
\end{lemma} 
\begin{proof} Consider the following commutative diagram of $\Lambda(A_{\mathfrak{p}})-$modules
\begin{center}
    \begin{tikzcd}
	0 & {\mathfrak{p} N} & N & {N/\mathfrak{p} N} & 0 \\
	0 & {\mathfrak{p} N^{\prime}} & {N^{\prime}} & {N^{\prime}/\mathfrak{p} N^{\prime}} & 0
	\arrow[from=1-1, to=1-2]
	\arrow[from=1-2, to=1-3]
	\arrow["{\phi}", from=1-2, to=2-2]
	\arrow[from=1-3, to=1-4]
	\arrow["\varphi", from=1-3, to=2-3]
	\arrow[from=1-4, to=1-5]
	\arrow["{\psi}", from=1-4, to=2-4]
	\arrow[from=2-1, to=2-2]
	\arrow[from=2-2, to=2-3]
	\arrow[from=2-3, to=2-4]
	\arrow[from=2-4, to=2-5]
\end{tikzcd}
\end{center} where $\phi, \psi$ are induced by $\varphi$. By snake lemma, we deduce that the sequence
\begin{align*}
    0 \rightarrow\text{ker}(\phi) \rightarrow \text{ker}(\varphi) \xrightarrow{f} \text{ker}(\psi) \xrightarrow{\delta} \text{coker}(\phi) \xrightarrow{g} \text{coker}(\varphi) \rightarrow \text{coker}(\psi) \rightarrow 0
\end{align*} is exact. By assumption, we know that $\text{ker}(\varphi), \text{coker}(\varphi)$ are finite. Hence, we immediately have that $\text{coker}(\psi)$ and $\text{im}(f)$ are finite. Furthermore, we deduce that
\begin{align*}
    \# \frac{\text{ker}(\psi)}{\text{im}(f)} = \# \frac{\text{ker}(\psi)}{\text{ker}(\delta)}  \leq \# \text{coker}(\phi) \leq \#\text{coker}(\varphi).
\end{align*} The last inequality holds because we have the following surjection
\begin{align*}
    \operatorname{coker}(\varphi) = N^{\prime}/\varphi(N) &\twoheadrightarrow \operatorname{coker}(\phi)=\mathfrak{p}N^{\prime}/\mathfrak{p}\varphi(N) \\
    [a] & \longrightarrow [\pi a]
\end{align*} where $\pi$ is a uniformizer of $\mathfrak{p}$.
Therefore, we conclude that $\text{ker}(\psi)$ is also finite and $\psi$ is a pseudo-isomorphism.
\end{proof}\noindent Recall that there is an isomorphism $\Lambda(\mathbb{Z}_p) \simeq \mathbb{Z}_p\left[\left[T\right]\right]$ (cf. \cite[\textbf{Theorem 7.1}]{bib3}) and the same proof generalizes verbatim to $A_{\mathfrak{p}}$: $\Lambda(A_{\mathfrak{p}}) \simeq A_{\mathfrak{p}}\left[\left[T\right]\right]$.
\begin{theorem} \label{theorem 4.6} Let $M$ be a $\Lambda(A_{\mathfrak{p}})-$module such that it is $\mathfrak{p}-$primary. Then the following statements are equivalent.
\begin{enumerate}[label={\arabic{enumi}}.]
    \item Denote $M^{\vee}$ by $N$. The $\Lambda(A_{\mathfrak{p}})-$module $N$ is finitely generated and torsion such that its $\mu-$Iwasawa invariant vanishes.
    \item $M[\mathfrak{p}]$ is finite.
\end{enumerate} Moreover, if one of the statements above is satisfied, then $\lambda \leq \operatorname{dim}_{A/\mathfrak{p}} (M[\mathfrak{p}])$.
\end{theorem}
\begin{proof}  $(1) \Longrightarrow (2):$ By assumption, we have a pseudo-isomorphism
\begin{align*}
    \varphi: N \longrightarrow A_{\mathfrak{p}}^{\oplus \lambda}
\end{align*} where $\lambda$ is the $\lambda-$invariant of $N$. By $\textbf{Lemma }\ref{Lemma 4.4}$ and $\textbf{Lemma }\ref{Lemma 4.5}$, we further deduce that
\begin{align*}
    \left(M[\mathfrak{p}]\right)^{\vee} \simeq N/\mathfrak{p}N \simeq \left(\frac{A_{\mathfrak{p}}}{\mathfrak{p}A_{\mathfrak{p}}}\right)^{\oplus \lambda}
\end{align*} which is finite. Since there is a canonical isomorphism
\begin{align*}
    M[\mathfrak{p}] \simeq (\left(M[\mathfrak{p}]\right)^{\vee} )^{\vee},
\end{align*} we conclude that $M[\mathfrak{p}]$ is finite. \\[5mm]
$(2) \Longrightarrow (1):$ We consider the following exact sequence
\begin{align*}
    0 \rightarrow M[\mathfrak{p}] \rightarrow M[\mathfrak{p}^{n+1}] &\rightarrow M[\mathfrak{p}^n], \\
    m&\longmapsto \pi m.
\end{align*} where $\pi$ is a fixed uniformizer of $A_{\mathfrak{p}}$. We deduce by induction on $n$ that $M[\mathfrak{p}^n]$ is finite for all $n$. Furthermore, we have
\begin{align*}
    N  \simeq \varprojlim_n \text{Hom}_{A_{\mathfrak{p}}}( M[\mathfrak{p^n}],F_{\mathfrak{p}}/A_{\mathfrak{p}}).
\end{align*}
Since the module $ M[\mathfrak{p^n}]$ is finite for all $n$ and so $\text{Hom}_{A_{\mathfrak{p}}}( M[\mathfrak{p^n}],F_{\mathfrak{p}}/A_{\mathfrak{p}})$ is finite, $N$ is compact. To see $N$ is finitely generated, we only need to show $N/(T,\mathfrak{p})$ is finite by Nakayama's lemma (cf. $\cite[\textbf{Lemma 13.16}]{bib3}$). But $N/(T,\mathfrak{p})$ is a quotient of $N/\mathfrak{p}N$ which is $(M[\mathfrak{p}])^{\vee}$ by $\textbf{Lemma }\ref{Lemma 4.4}$ and it is finite by assumption. To finish the proof of this direction, we need to show that the $\mu-$invariant of $N$ vanishes and $N$ is torsion. This is an easy consequence of $\cite[\textbf{Proposition 3.3}]{bib1}$ and the fact that $N/\mathfrak{p}N$ is finite. \\[5mm] If we assume that $\text{(1)}$ is true, we have a pseudo-isomorphism $\varphi:N \longrightarrow A_{\mathfrak{p}}^{\oplus \lambda}$. By the same argument as in the proof of $\textbf{Theorem }\ref{theorem 3.12}$, we have
 \begin{align*}
     N \simeq A_{\mathfrak{p}}^{\oplus \lambda} \oplus \text{ker}(\varphi)
 \end{align*} and therefore, we deduce that
 \begin{align*}
    \lambda \leq \text{dim}_{A/\mathfrak{p}} N/\mathfrak{p}N = \text{dim}_{A/\mathfrak{p}} (M[\mathfrak{p}])^{\vee} = \text{dim}_{A/\mathfrak{p}} M[\mathfrak{p}].
 \end{align*} Therefore, we conclude the proof of the theorem.
\end{proof}

\begin{center}
    \section{Fine Selmer groups and main results}\label{sec4}
\end{center}

\justifying

\noindent In this section, we define the fine Selmer group and the residual fine Selmer group associated to a Drinfeld module over $F$ and we study their properties. We aim to give a proof for $\textbf{Theorem }\ref{theorem 1.3}$ based on the results of Sections 2, 3 and 4. \\[5mm] For a fixed Drinfeld module $\phi$ over a global function field $F$, we consider the set $S$ containing places of $F$ where
\begin{align*}
    S = \{ \mathfrak{p}, \infty\} \cup \{ \text{ places corresponding to the bad reductions of } \phi \}.
\end{align*}  Denote the maximal separable extension of $F$ in which the places are ramified outside $S$ by $F_S$ and denote the constant $\mathbb{Z}_p-$extension of $F$ by $\mathcal{F}$. With respect to this setup, we have the following tower of field extensions
\begin{align} \label{twof}
    F \subseteq \mathcal{F}  \subseteq F_S \subseteq F^{\text{sep}} = \mathcal{F}^{\text{sep}}.
\end{align} The second inclusion in $(\ref{twof})$ holds because constant extensions of a function field are unramified (cf. \cite[\textbf{Theorem 3.6.3 (a)}]{bib8}) and in particular, it is contained in $F_S$. The last equality in $(\ref{twof})$ is valid because $\mathcal{F}/F$ is separable since for each $n \geq 1$, we have $F_n = F(\alpha_n)$ where the minimal polynomial of $\alpha_n$ is $x^{p^n} - x$ (cf. \cite[\textbf{Lemma 3.6.2}]{bib8}) and it is separable over $F$. Therefore, we may deduce the following exact sequence of Galois cohomologies
\begin{align} \label{inf-res}
   0 \rightarrow H^1(F_S/\mathcal{F},\phi[\mathfrak{p}^{\infty}]) \xrightarrow{ \textbf{inf}}H^1(\mathcal{F}^{\text{sep}}/\mathcal{F},\phi[\mathfrak{p}^{\infty}]) \xrightarrow{\textbf{res}}H^1(\mathcal{F}^{\text{sep}}/F_S,\phi[\mathfrak{p}^{\infty}])^{\operatorname{Gal}(F_S/\mathcal{F})}
\end{align}which is the inflation-restriction sequence (cf. \cite[\textbf{VII,\S 6,Proposition 4}]{bib12}). Notice that $\operatorname{Gal}(F^{\operatorname{sep}}/F_S)$ acts trivally on $\phi[\mathfrak{p}^{\infty}]$ (cf. $\cite[\textbf{Theorem 6.3.1}]{bib2}$). Now, given an arbitrary place $v$ in $F$, we denote the places of $\mathcal{F}$ above $v$ by $v(\mathcal{F})$ which is a finite set. For each $w$ in $v(\mathcal{F})$, we denote the union of the completion of $F^{\prime}$ at $w$ where $F^{\prime}$ ranges over all the finite extensions of $F$ contained in $\mathcal{F}$ by $\mathcal{F}_w$. Then we put
\begin{align} \label{Jv}
    J_v(\phi[\mathfrak{p}^{\infty}]/\mathcal{F}) := \prod_{w \in v(\mathcal{F})} H^1(\mathcal{F}_w^{\text{sep}}/\mathcal{F}_w, \phi[\mathfrak{p}^{\infty}]).
\end{align}We further consider the map
\begin{align} \label{phi}
    \Phi: H^{1}(F_S/\mathcal{F},\phi[\mathfrak{p}^{\infty}]) & \longrightarrow \bigoplus_{v \in S} J_\mathfrak{p}(\phi[\mathfrak{p}^{\infty}]/\mathcal{F})
\end{align} where $\Phi$ is the composition of the map $\textbf{inf}$ as in $(\ref{inf-res})$ to $H^1(\mathcal{F}^{\text{sep}}/\mathcal{F},\phi[\mathfrak{p}^{\infty}])$ and the restriction map from $H^1(\mathcal{F}^{\text{sep}}/\mathcal{F},\phi[\mathfrak{p}^{\infty}])$ to each $ H^1(\mathcal{F}_w^{\text{sep}}/\mathcal{F}_w, \phi[\mathfrak{p}^{\infty}])$, respectively. \\[5mm]
\begin{definition} Let $\phi$ be a Drinfeld module over $F$. Let $\mathfrak{p}$ be a non-zero prime ideal of $A$ as in \normalfont{(\ref{A})}. Let $\mathcal{F}$ be the constant $\mathbb{Z}_p-$extension of $F$. Then the fine Selmer group associated to $\phi$ is defined to be
\begin{align*}
    \operatorname{Sel}^S_0(\phi[\mathfrak{p}^{\infty}]/\mathcal{F}) := \operatorname{ker}(\Phi)
\end{align*} where $S$ is the set as described in $(\ref{S})$.
\end{definition} \noindent Now we equip a $\Lambda(A_{\mathfrak{p}})-$module structure on $\operatorname{Sel}_0^S(\phi[\mathfrak{p}^{\infty}]/\mathcal{F})$ and its Pontryagin dual  where $A_{\mathfrak{p}}$ is the completion of $A$ as in $(\ref{A})$ with respect to $\mathfrak{p}$ (cf. \cite[\textbf{Theorem 7.1}]{bib3}) and $  \Lambda(A_{\mathfrak{p}}) =A_{\mathfrak{p}}\left[\left[\operatorname{Gal}(\mathcal{F}/F)\right]\right] $. Consider the isomorphism
\begin{align*}
   A_{\mathfrak{p}}\left[\left[\operatorname{Gal}(\mathcal{F}/F)\right]\right] \simeq \varprojlim_n A_{\mathfrak{p}}[\operatorname{Gal}(F_n/F)]
\end{align*} Furthermore, there is an isomorphism for the fine Selmer group
\begin{align*}
    \text{Sel}_0^{S}(\phi[\mathfrak{p}^{\infty}]/\mathcal{F}) \simeq \varinjlim_n   \text{Sel}_0^{S}(\phi[\mathfrak{p}^{\infty}]/F_n)
\end{align*} (cf. \cite[\textbf{(44), (45)}]{bib6}). For each $n \geq 0$, there is an $A_{\mathfrak{p}}[\operatorname{Gal}(F_n/F)]-$module structure on  $\text{Sel}_0^{S}(\phi[\mathfrak{p}^{\infty}]/F_n)$ through the following Galois action
\begin{align*}
    \operatorname{Gal}(F_n/F) \times \text{Sel}_0^{S}(\phi[\mathfrak{p}^{\infty}]/F_n) & \longrightarrow \text{Sel}_0^{S}(\phi[\mathfrak{p}^{\infty}]/F_n) \\
    (\sigma, f) & \longmapsto \tau  \longmapsto \Tilde{\sigma} f(\Tilde{\sigma}^{-1}\tau \Tilde{\sigma}),
\end{align*} where $\Tilde{\sigma}$ is any lift in $\operatorname{Gal}(F^{\text{sep}}/F)$ of $\sigma$. This Galois action is independent of the choice of $\Tilde{\sigma}$ and we therefore equip the fine Selmer group with a $\Lambda(A_{\mathfrak{p}})-$module structure. Moreover, since $\text{Sel}_0^{S}(\phi[\mathfrak{p}^{\infty}]/\mathcal{F})$ is $\mathfrak{p}-$primary, we set the Pontryagin dual of $\text{Sel}_0^{S}(\phi[\mathfrak{p}^{\infty}]/\mathcal{F})$ to be
\begin{align*}
    Y^S(\phi[\mathfrak{p}^{\infty}]/\mathcal{F}) := \text{Hom}_{A_{\mathfrak{p}}}(\text{Sel}_0^{S}(\phi[\mathfrak{p}^{\infty}]/\mathcal{F}), F_{\mathfrak{p}}/A_{\mathfrak{p}}).
\end{align*} The $\Lambda(A_{\mathfrak{p}})-$module structure on $Y^S(\phi[\mathfrak{p}^{\infty}]/\mathcal{F})$ is constructed in the following way
\begin{align*}
   \Lambda(A_{\mathfrak{p}}) \times Y^S(\phi[\mathfrak{p}^{\infty}]/\mathcal{F}) & \longrightarrow Y^S(\phi[\mathfrak{p}^{\infty}]/\mathcal{F}) \\
    (\gamma, \psi) & \longmapsto f\longmapsto \psi(\gamma^{\prime} f)
\end{align*} where $\gamma\prime$ is the image of $\gamma$ under the isomorphism 
\begin{align*}
     A_{\mathfrak{p}}\left[\left[\operatorname{Gal}(\mathcal{F}/F)\right]\right] &\longrightarrow A_{\mathfrak{p}}\left[\left[\operatorname{Gal}(\mathcal{F}/F)\right]\right], \\
     \sum a \sigma_a & \longmapsto \sum a\sigma_a^{-1}.
\end{align*} For the proof of $\textbf{Theorem }\ref{theorem 1.3}$, we further give the definition of residual fine Selmer groups and study their properties. We recall the map $\Phi$ in $(\ref{phi})$ and similarly, we consider the map $\Psi$ where we replace $\phi[\mathfrak{p}^{\infty}]$ in $(\ref{inf-res})$ and $(\ref{Jv})$ with $\phi[\mathfrak{p}]:$
\begin{align} \label{Psi}
     \Psi: H^{1}(F_S/\mathcal{F},\phi[\mathfrak{p}]) & \longrightarrow \bigoplus_{v \in S} J_v(\phi[\mathfrak{p}]/\mathcal{F}).
\end{align}
\begin{definition} \label{definition 5.2} Let $\phi$ be a Drinfeld module over $F$. Let $\mathfrak{p}$ be a non-zero prime ideal of $A$ as in $(\ref{A})$. Let $\mathcal{F}$ be the constant $\mathbb{Z}_p-$extension of $F$. Then the residual fine Selmer group associated to $\phi$ is defined to be
\begin{align*}
    \emph{Sel}^S_0(\phi[\mathfrak{p}]/\mathcal{F}) := \emph{ker}(\Psi)
\end{align*} where $S$ is the set as described in $(\ref{S})$ and $\Psi$ is the map as in $(\ref{Psi})$.
\end{definition}
\noindent There is an action of $\operatorname{Gal}(\mathcal{F}^{\operatorname{sep}}/\mathcal{F})$ on both $\phi[\mathfrak{p}]$ and $\phi[\mathfrak{p}^{\infty}]$ induced by the natural action of $\operatorname{Gal}(\mathcal{F}^{\operatorname{sep}}/F)\simeq \operatorname{Gal}(F^{\operatorname{sep}}/F)$ on $\phi[\mathfrak{p}]$ and $\phi[\mathfrak{p}^{\infty}]$. In addition, we restrict this action to the decomposition group $\operatorname{Gal}(\mathcal{F}_w^{\operatorname{sep}}/\mathcal{F}_w)$ where $w$ is a place above $v$ in $S$. Denote the field $A_{\mathfrak{p}}/\mathfrak{p}A_{\mathfrak{p}}$ by $\mathbb{F}_{\mathfrak{p}}$. We now derive the following short exact sequence of Galois cohomologies with respect to the short exact sequence $(\ref{galoiscohomology})$
\begin{align} \label{kummer}
    0 \rightarrow H^0(\mathcal{F}_w^{\operatorname{sep}}/\mathcal{F}_w,\phi[\mathfrak{p}^{\infty}])\otimes_{A_{\mathfrak{p}}} \mathbb{F}_{\mathfrak{p}} \rightarrow H^1(\mathcal{F}_w^{\operatorname{sep}}/\mathcal{F}_w, \phi[\mathfrak{p}]) \xrightarrow{\varphi} H^1(\mathcal{F}_w^{\operatorname{sep}}/\mathcal{F}_w, \phi[\mathfrak{p}^{\infty}])[\mathfrak{p}]\rightarrow 0.
\end{align} Given any $v$ in $S$ and $w$ in $v(\mathcal{F})$, we construct the following map
\begin{align*}
    \prod_{w \in v(\mathcal{F})} h_w: \prod_{w \in v(\mathcal{F})}  H^1( \operatorname{Gal}(\mathcal{F}_w^{\operatorname{sep}}/\mathcal{F}_w), \phi[\mathfrak{p}]) \rightarrow \prod_{w \in v(\mathcal{F})} H^1( \operatorname{Gal}(\mathcal{F}_w^{\operatorname{sep}}/\mathcal{F}_w), \phi[\mathfrak{p}^{\infty}])[\mathfrak{p}]
\end{align*} where each $h_w$ corresponds to the map $\varphi$ defined in $(\ref{kummer})$, respectively. We repeat the procedure for each $v$ in $S$ and obtain the map
\begin{align} \label{h}
    h: \bigoplus_{v \in S} J_v(\phi[\mathfrak{p}]/\mathcal{F}) \rightarrow \bigoplus_{v \in S} J_v(\phi[\mathfrak{p}^{\infty}]/\mathcal{F})[\mathfrak{p}].
\end{align} On the other hand, since the natural action of $\operatorname{Gal}(F^{\operatorname{sep}}/F)$ on $\phi[\mathfrak{p}]$ and $\phi[\mathfrak{p}^{\infty}]$ restricting to the subgroup $\operatorname{Gal}(F^{\operatorname{sep}}/F_S)$ is trivial (cf. $\cite[\textbf{Theorem 6.3.1}]{bib2}$), there is a well-defined action of $\operatorname{Gal}(F_S/F)$ on $\phi[\mathfrak{p}]$ and $\phi[\mathfrak{p}^{\infty}]$ induced by the natural action of $\operatorname{Gal}(F^{\operatorname{sep}}/F)$. Similarly to the construction of the map $\varphi$ ($\ref{kummer}$), we obtain the following surjective map
\begin{align*}
    \beta: H^1(F_S/\mathcal{F},\phi[\mathfrak{p}]) \rightarrow H^1(F_S/\mathcal{F},\phi[\mathfrak{p}^{\infty}])[\mathfrak{p}].
\end{align*} Furthermore, we restrict $\beta$ to the map
\begin{align*}
    \gamma:  \operatorname{Sel}_0^S(\phi[\mathfrak{p}]/\mathcal{F}) \longrightarrow \operatorname{Sel}_0^S(\phi[\mathfrak{p}^{\infty}]/\mathcal{F})[\mathfrak{p}].
\end{align*} We therefore obtain the following commutative diagram
\begin{center} \label{diagram}
    \begin{tikzcd}
	& {\operatorname{Sel}_0^S(\phi[\mathfrak{p}]/\mathcal{F})} & {H^1(F_S/\mathcal{F},\phi[\mathfrak{p}])} & {\operatorname{im}(\Psi)} & 0 \\
	0 & {\operatorname{Sel}_0^S(\phi[\mathfrak{p}^{\infty}]/\mathcal{F})}[\mathfrak{p}] & {H^1(F_S/\mathcal{F},\phi[\mathfrak{p}^{\infty}])[\mathfrak{p}]} & {\bigoplus_{v \in S} J_v(\phi[\mathfrak{p}^{\infty}]/\mathcal{F})[\mathfrak{p}]}
	\arrow[from=1-2, to=1-3]
	\arrow["\gamma", from=1-2, to=2-2]
	\arrow[from=1-3, to=1-4]
	\arrow["\beta", from=1-3, to=2-3]
	\arrow[from=1-4, to=1-5]
	\arrow["{h^{\prime}}", from=1-4, to=2-4]
	\arrow[from=2-1, to=2-2]
	\arrow[from=2-2, to=2-3]
	\arrow[from=2-3, to=2-4]
\end{tikzcd}
\end{center} where $\Psi$ is the map as in $(\ref{Psi})$ and $h^{\prime}$ is the restriction of the map $h$ $(\ref{h})$ to $\operatorname{im}(\Psi)$.
\begin{lemma} \label{lemma 5.3}With respect to the notations above, the map $\gamma$
\begin{align*}
    \gamma: \emph{Sel}_0^S(\phi[\mathfrak{p}]/\mathcal{F}) \longrightarrow \emph{Sel}_0^S(\phi[\mathfrak{p}^{\infty}]/\mathcal{F})[\mathfrak{p}]
\end{align*} is a pseudo-isomorphism. Furthermore, the kernel and cokernel of $\gamma$ satisfy
\begin{align*}
   & \# \operatorname{ker} (\gamma) \leq \# \left(H^0(F_S/\mathcal{F}, \phi[\mathfrak{p}^{\infty}])\otimes_{A_{\mathfrak{p}}}\mathbb{F}_{\mathfrak{p}}\right), \\
    &  \# \operatorname{coker} (\gamma) \leq \# \prod_{w\in S(\mathcal{F})}\left(H^0(\mathcal{F}_w^{\operatorname{sep}}/\mathcal{F}_w, \phi[\mathfrak{p}^{\infty}])\otimes_{A_{\mathfrak{p}}}\mathbb{F}_{\mathfrak{p}}\right).
\end{align*}
\end{lemma}
\begin{proof} By $\emph{Snake Lemma}$, the commutative diagram above yields that
\begin{align*}
    0 \rightarrow \operatorname{ker} (\gamma) \rightarrow \operatorname{ker}(\beta) \rightarrow \operatorname{ker} (h^{\prime}) \rightarrow \operatorname{coker} (\gamma) \rightarrow 0
\end{align*} since the map $\beta$ is surjective. Therefore, we conclude that
\begin{align*}
    & \# \operatorname{ker} (\gamma) \leq \# \operatorname{ker} (\beta) = \# \left(H^0(F_S/\mathcal{F}, \phi[\mathfrak{p}^{\infty}])\otimes_{A_{\mathfrak{p}}}\mathbb{F}_{\mathfrak{p}}\right), \\
    &  \# \operatorname{coker} (\gamma) \leq \operatorname{ker} (h^{\prime}) \leq \# \operatorname{ker}(h)\leq \# \prod_{w\in S(\mathcal{F})}\left(H^0(\mathcal{F}_w^{\operatorname{sep}}/\mathcal{F}_w, \phi[\mathfrak{p}^{\infty}])\otimes_{A_{\mathfrak{p}}}\mathbb{F}_{\mathfrak{p}}\right).
\end{align*} Since $\phi[\mathfrak{p}^{\infty}] \simeq (F_{\mathfrak{p}}/A_{\mathfrak{p}})^{\oplus r}$ where $r$ is the rank of $\phi$ ($\ref{rankr isomorphism}$), we further deduce that
\begin{align*}
     \# \left(H^0(F_S/\mathcal{F}, \phi[\mathfrak{p}^{\infty}])\otimes_{A_{\mathfrak{p}}}\mathbb{F}_{\mathfrak{p}}\right) \leq r\# \mathbb{F}_\mathfrak{p},  \#\left( \prod_{w\in S(\mathcal{F})}\left(H^0(\mathcal{F}_w^{\operatorname{sep}}/\mathcal{F}_w, \phi[\mathfrak{p}^{\infty}])\otimes_{A_{\mathfrak{p}}}\mathbb{F}_{\mathfrak{p}}\right)\right) \leq r\#S(\mathcal{F})\# \mathbb{F}_\mathfrak{p}.
\end{align*}
This concludes the proof of the lemma.
\end{proof} \noindent The following proposition gives the crucial equivalence between the finiteness of residual fine Selmer group and the statement $(1)$ of $\textbf{Theorem }\ref{theorem 1.3}$.
\begin{proposition} \label{proposition 5.4}With respect to the notations above, the following statements are equivalent:
\begin{enumerate}[label={\arabic{enumi}}.]
    \item The Pontryagin dual $Y^S(\phi[\mathfrak{p}^{\infty}]/\mathcal{F})$ is a finitely generated and torsion $\Lambda(A_{\mathfrak{p}})-$module such that its $\mu-$Iwasawa invariant vanishes.
    \item The group $\emph{Sel}_0^S(\phi[\mathfrak{p}]/\mathcal{F})$ is finite. 
\end{enumerate} Furthermore, if one of the assertions above holds, then the $\lambda-$invariant of $Y^S(\phi[\mathfrak{p}^{\infty}]/\mathcal{F})$ satisfies
\begin{align*}
     \lambda\leq {\dim}_{\mathbb{F}_\mathfrak{p}} \text{\normalfont\ Sel}_0^S(\phi[\mathfrak{p}]/\mathcal{F}) + \sum_{w \in S(\mathcal{F})} {\dim}_{\mathbb{F}_{\mathfrak{p}}}(H^0(\mathcal{F}_w^{\operatorname{sep}}/\mathcal{F}_w, \phi[\mathfrak{p^{\infty}}]\otimes_{ A_{\mathfrak{p}}}\mathbb{F}_{\mathfrak{p}}))
\end{align*} where $\mathbb{F}_{\mathfrak{p}}$ is the residue field $ A_{\mathfrak{p}}/\mathfrak{p}A_{\mathfrak{p}}$.
\end{proposition}
\begin{proof} We denote the fine Selmer group $\operatorname{Sel}_0^S(\phi[\mathfrak{p}^{\infty}]/\mathcal{F})$ by $M$. Since $M$ is $\mathfrak{p}-$primary and it is a $\Lambda(A_{\mathfrak{p}})-$module. We know that $M[\mathfrak{p}]$ is finite if and only if the Pontryagin dual of $M$ is finitely generated and torsion module over $\Lambda(A_{\mathfrak{p}})$ such that its $\mu-$invariant vanishes as a result of $\textbf{Theorem }\ref{theorem 4.6}$. On the other hand, we have that $M[\mathfrak{p}]$ is finite if and only if $\operatorname{Sel}^S_0(\phi[\mathfrak{p}]/\mathcal{F})$ is finite by $\textbf{Lemma }\ref{lemma 5.3}$. This concludes the proof for the equivalence in the statement. Furthermore, the last part of $\textbf{Theorem }\ref{theorem 4.6}$ gives that the $\lambda-$invariant of the Pontryagin dual of $M$ satisfies
\begin{align*}
    \lambda \leq \operatorname{dim}_{\mathbb{F}_{\mathfrak{p}}} M[\mathfrak{p}].
\end{align*} Since $M[\mathfrak{p}] = \operatorname{Sel}_0^S(\phi[\mathfrak{p}^{\infty}]/\mathcal{F})[\mathfrak{p}]$, we further conclude from $\textbf{Lemma }\ref{lemma 5.3}$ that
\begin{align*}
    \lambda \leq \operatorname{dim}_{\mathbb{F}_{\mathfrak{p}}} M[\mathfrak{p}] &\leq \operatorname{dim}_{\mathbb{F}_{\mathfrak{p}}} \operatorname{Sel}_0^S(\phi[\mathfrak{p}]/\mathcal{F}) + \operatorname{dim}_{\mathbb{F}_{\mathfrak{p}}} \operatorname{coker} (\gamma)  \\
    &\leq \operatorname{dim}_{\mathbb{F}_{\mathfrak{p}}} \operatorname{Sel}_0^S(\phi[\mathfrak{p}]/\mathcal{F}) + \sum_{w\in S(\mathcal{F})}\operatorname{dim}_{\mathbb{F}_{\mathfrak{p}}}\left(H^0(\mathcal{F}_w^{\operatorname{sep}}/\mathcal{F}_w, \phi[\mathfrak{p}^{\infty}])\otimes_{A_{\mathfrak{p}}}\mathbb{F}_{\mathfrak{p}}\right).
\end{align*} This concludes the proof of the proposition.
\end{proof}
\begin{proposition} \label{proposition 5.5} With respect to the notations above, the residual fine Selmer group $\emph{Sel}_0^S(\phi[\mathfrak{p}]/\mathcal{F})$ is finite
\end{proposition}
\begin{proof} We consider the finite extension $L/F$ where $L = F(\phi[\mathfrak{p}])$ and we denote the constant $\mathbb{Z}_p-$extension of $L$ by $\mathcal{L}$. Furthermore, we consider the inflation-restriction sequence
\begin{align} \label{inf-res1}
    0 \rightarrow H^1(\mathcal{L}/\mathcal{F},\phi[\mathfrak{p}])\xrightarrow {\textbf{inf}} H^1(F_S/\mathcal{F},\phi[\mathfrak{p}]) \xrightarrow{\textbf{res}}  H^1(F_S/\mathcal{L},\phi[\mathfrak{p}]).
\end{align} Since $\phi[\mathfrak{p}]$ is trivial under the action of $\operatorname{Gal}(\mathcal{L}/\mathcal{F})$, we conclude that
\begin{align*}
    H^1(F_S/\mathcal{L},\phi[\mathfrak{p}]) \simeq \operatorname{Hom}_{\operatorname{Grps}}(\operatorname{Gal}(F_S/\mathcal{L}), \mathbb{F}_{\mathfrak{p}}^{\oplus r}).
\end{align*}  Denote the set of places in $L$ above $S$ by $S(L)$ and denote the inverse limit of
\begin{align*}
   \varprojlim_n \operatorname{Gal}(H_p^{S(L)}(L_n)/L_n)
\end{align*} by $X_{S(L)}(\mathcal{L})$. For any $f$ in $\operatorname{Sel}_0^S(\phi[\mathfrak{p}]/\mathcal{F})$ contained in $ H^1(F_S/\mathcal{F},\phi[\mathfrak{p}])$, there is a unique factorization of $\operatorname{res}(f)$
\begin{align} \label{infres}
    \operatorname{res}(f): \frac{X_{S(L)}(\mathcal{L})}{pX_{S(L)}(\mathcal{L})} \longrightarrow \mathbb{F}_{\mathfrak{p}}^{\oplus r}
\end{align} because $\mathbb{F}_{\mathfrak{p}}^{\oplus r}$ is an abelian group with characteristic $p$ and the crossed homomorphism $f$ is taken from $\operatorname{Sel}_0^S(\phi[\mathfrak{p}]/\mathcal{F})$. Since the quotient group $\frac{X_{S(L)}(\mathcal{L})}{pX_{S(L)}(\mathcal{L})}$ is finite as a result of $\textbf{Theorem }\ref{theorem 3.12}$, the image of $\textbf{res}$ in $(\ref{inf-res1})$ restricting to $\operatorname{Sel}_0^S(\phi[\mathfrak{p}]/\mathcal{F})$ is finite. On the other hand, the cohomology group $H^1(\mathcal{L}/\mathcal{F},\phi[\mathfrak{p}])$ is finite because the Galois group $\operatorname{Gal}(\mathcal{L}/\mathcal{F})$ is finite. Hence, we conclude that $\operatorname{Sel}_0^S(\phi[\mathfrak{p}]/\mathcal{F})$ is finite.
\end{proof}\noindent
Now we give the proof for $\textbf{Theorem }\ref{theorem 1.3}$.
\begin{proof} The proof for $(1)$ in $\textbf{Theorem }\ref{theorem 1.3}$ follows as a direct consequence of $\textbf{Proposition }\ref{proposition 5.4}$ and $\textbf{Proposition }\ref{proposition 5.5}$. Since $(1)$ of $\textbf{Theorem }\ref{theorem 1.3}$, there is a pseudo-isomorphism 
\begin{align*}
    \varphi: Y^S(\phi[\mathfrak{p}^{\infty}]/\mathcal{F}) \longrightarrow (A_{\mathfrak{p}})^{\oplus \lambda}
\end{align*} where $\lambda$ is the $\lambda-$invariant of  $Y^S(\phi[\mathfrak{p}^{\infty}]/\mathcal{F})$. Since $A_{\mathfrak{p}}$ is a P.I.D,  we repeat the argument as in the proof of $\textbf{Theorem }\ref{theorem 3.12}$ and we have 
\begin{align*}
    Y^S(\phi[\mathfrak{p}^{\infty}]/\mathcal{F}) \simeq (A_{\mathfrak{p}})^{\oplus \lambda} \oplus \operatorname{ker}(\varphi),
\end{align*} which concludes the proof of $(2)$ in $\textbf{Theorem }\ref{theorem 1.3}$. Finally, $(3)$ in $\textbf{Theorem }\ref{theorem 1.3}$ follows directly from the last part of $\textbf{Proposition }\ref{proposition 5.4}$.
\end{proof}   
\begin{center}

   {\normalfont\bibliography{sn-bibliography}}
\end{center}
\end{document}